\documentclass[11pt]{amsart}

\usepackage{amsmath,amssymb,latexsym,soul,cite,mathrsfs}

\usepackage{color,enumitem,graphicx}
\usepackage[colorlinks=true,urlcolor=blue,
citecolor=red,linkcolor=blue,linktocpage,pdfpagelabels,
bookmarksnumbered,bookmarksopen]{hyperref}
\usepackage[english]{babel}

\usepackage[left=2.9cm,right=2.9cm,top=2.8cm,bottom=2.8cm]{geometry}
\usepackage[hyperpageref]{backref}

\usepackage[colorinlistoftodos]{todonotes}
\makeatletter
\providecommand\@dotsep{5}
\def\listtodoname{List of Todos}
\def\listoftodos{\@starttoc{tdo}\listtodoname}
\makeatother

\numberwithin{equation}{section}

\newtheorem{theorem}{Theorem}[section]
\newtheorem{proposition}[theorem]{Proposition}
\newtheorem{lemma}[theorem]{Lemma}
\newtheorem{corollary}[theorem]{Corollary}

\newtheorem{claim}[theorem]{Claim}

\begin{document}
	
	\title[Existence of solution for a class of elliptic equation ... ]
	{Existence of solution for a class of elliptic equation with discontinuous nonlinearity and asymptotically linear }

	\author{Claudianor O. Alves}
	\author{Geovany F. Patricio}

	\address[Claudianor O. Alves]{\newline\indent Unidade Acad\^emica de Matem\'atica
		\newline\indent 
		Universidade Federal de Campina Grande,
		\newline\indent
		58429-970, Campina Grande - PB - Brazil}
	\email{\href{mailto:coalves@mat.ufcg.edu.br}{coalves@mat.ufcg.edu.br}}
	
	\address[Geovany F. Patricio]
	{\newline\indent Unidade Acad\^emica de Matem\'atica
		\newline\indent 
		Universidade Federal de Campina Grande,
		\newline\indent
		58429-970, Campina Grande - PB - Brazil}
	\email{\href{fernandes.geovany@yahoo.com.br}{fernandes.geovany@yahoo.com.br}}

	\pretolerance10000
	
	
	\begin{abstract}
		\noindent This paper concerns the existence of a nontrivial solution for the following problem
		\begin{equation}
		\left\{\begin{aligned}
		-\Delta u + V(x)u & \in \partial_u F(x,u)\;\;\mbox{a.e. in}\;\;\mathbb{R}^{N},\nonumber \\
		u \in H^{1}(\mathbb{R}^{N}), 
		\end{aligned}
		\right.\leqno{(P)}
		\end{equation} 
where $F(x,t)=\int_{0}^{t}f(x,s)\,ds$, $f$ is a discontinuous function and asymptotically linear at infinity, $\lambda=0$ is in a spectral gap of $-\Delta+V$, and  $\partial_t F$ denotes the generalized gradient of $F$ with respect to variable $t$. Here, by employing Variational Methods for Locally Lipschitz Functionals, we establish the existence of solution when $f$ is periodic and non periodic.	
\end{abstract}

\maketitle

\section{Introduction}
In this paper we study the existence of nontrivial solution for the following class of elliptic problems
\begin{equation}
\left\{\begin{aligned}
-\Delta u + V(x)u & \in \partial_u F(x,u)\;\;\mbox{a.e. in}\;\;\mathbb{R}^{N},\nonumber \\
u \in H^{1}(\mathbb{R}^{N}), 
\end{aligned}
\right.\leqno{(P)}
\end{equation} 
where $F(x,t)=\int_{0}^{t}f(x,s)\,ds$, $f$ is a discontinuous function and asymptotically linear at infinity, $\lambda=0$ is in a spectral gap of $-\Delta+V$ and  $\partial_t F$ denotes the generalized gradient of $F$ with respect to variable $t$. 

The problem $(P)$ for the case where $f$ is a continuous function becomes 
\begin{equation}
	\left\{\begin{array}{l}
		-\Delta u + V(x)u=f(x,u)\;\;  \mbox{a.e. in}\;\;\mathbb{R}^{N},\nonumber \\
		u \in H^{1}(\mathbb{R}^{N}),
	\end{array}
	\right. \leqno{(P_1)}
\end{equation} 
and it has been studied for some authors. In \cite{G.B Li}, Li and Szulkin have improved the generalized link theorem obtained in Kryszewski and Szulkin \cite{Kryszewski} to establish the existence of solution for $(P_1)$ with $f$ being asymptotically linear at infinity and asymptotic to a $\mathbb{Z}^{N}$-periodic function. Motivated by results found in \cite{G.B Li}, some authors studied the problem $(P_1)$ with the same conditions on $V$ and supposing other conditions on $f$, but  $f$ still being asymptotically linear, see for example,  Ding and Lee \cite{Ding},  Chen and Dawei \cite{Shaowei}, Tang \cite{X1Tang}, Lin and Tang \cite{H1Tang}, Wu and Qin \cite{Qin1} and their references.

The main motivation of the present paper comes from the study found in Li and Szulkin \cite{G.B Li} and Alves and Patricio \cite{AG}. In \cite{AG}, the authors have studied the existence of nontrivial solution for problem $(P)$ for a class of superlinear problem where the nolinearity is a discontinuous function and $\lambda=0$ is in a spectral gap of $-\Delta+V$. In that paper, it was proved a  generalized link theorem for Locally Lipschitz functionals that improves the generalized link theorem  found in \cite{Kryszewski}, and after that, the authors used their link theorem to prove  the existence of nontrivial solution for $(P)$.

Hereafter, we assume the following conditions on $f$ and $V$: 
\begin{itemize}
	\item [(H1)] $V:\mathbb{R}^{N}\rightarrow \mathbb{R}$ is continuous, $\mathbb{Z}^{N}$-periodic and
	\begin{equation*}
	0 \notin \sigma(-\Delta +V).
	\end{equation*}
	In the sequel, $(\mu_{-1}, \mu_1)$ is the spectral gap containing 0 and $\mu_0:= \min\{-\mu_1,\mu_1\}$.
	\item [(H2)]   
	\begin{equation*}
	\lim_{t\rightarrow 0}\frac{f(x,t)}{t}=0 \quad \mbox{uniformly in} \quad x \in \mathbb{R}^N,
	\end{equation*}
	and for any $0<a<b<+\infty$, there is $M=M(a,b)>0$ such that 
	$$
	|f(x,t)| \leq M, \quad \forall x \in \mathbb{R}^N  \quad \mbox{and} \quad |t| \in [a,b].
	$$
	\item [(H3)] $f(x,t)= V_{\infty}(x)t+f_{\infty}(x,t)$ where $f_{\infty}(x,\cdot)$ is a measurable function defined on $\mathbb{R}^N \times \mathbb{R}$, $\frac{f_{\infty}(x,t)}{t}\rightarrow 0$ uniformly with respect to $x$ as $|t|\rightarrow +\infty$, $V_{\infty}$ is $\mathbb{Z}^{N}$-periodic and $V_{\infty}(x)\geq \mu,\;\forall\; x \in \mathbb{R}^{N}$, for some $\mu>\mu_{1}$. Moreover, the functions 	
	$$
	\underline{f}(x,t)= \lim_{r\downarrow 0} ess \inf\{f(x,s); |s-t|<r\}
	$$
	and
	$$
	\overline{f}(x,t)= \lim_{r\downarrow 0} ess \sup\{f(x,s); |s-t|<r\}.
	$$
	are $N$-measurable functions, see Chang \cite{Chang1, chang2, chang3} for more details.
	
	\item [(H4)] $0 \leq F(x,t) \leq \frac{1}{2} \rho t$ for all $\rho \in \partial_{t}F(x,t)$ and $t \in \mathbb{R}$, where
	\begin{equation*}
	F(x,t)=\int_{0}^{t}f(x,s) ds.
	\end{equation*}
	
	\item [(H5)] There exists $\delta \in (0, \mu_{0})$ such that if $\;\; \frac{\rho }{t} \geq \mu_{0}-\delta$ and $\rho\in  \partial_{t}F(x,t)$, then $\frac{1}{2}\rho t-F(x,t) \geq \delta.$
\end{itemize}

A nonlinearity $f$ that satisfies the conditions above is the following: Fixed $b>0$, let us consider the function 
	$$
f(x,t)= \left\{\begin{aligned}
& \mu t- \mu \arctan(t),\;\;\mbox{if}\;\; |t|\leq b \\
&  \mu t+\mu\;(\gamma-1)\;\arctan(t),\;\; \mbox{if}\;\; |t|> b
\end{aligned}
\right. 
$$
where $\mu>\mu_{1}$ and $0<\gamma$ is such that $\gamma < \dfrac{\mu_{0}}{\mu}$.

\vspace{0.5 cm}

Our main result is the following:
\begin{theorem} \label{Teorema1}(The periodic case)
	Assume $(H1)-(H5)$ and that $f$ is $\mathbb{Z}^{N}$-periodic. Then, the problem $(P)$ has a nontrivial solution.
\end{theorem}

The Theorem \ref{Teorema1} complements the study made in \cite{G.B Li} and \cite{AG} in the following sense: It complements \cite{G.B Li}, because in that paper the nonlinearity is continuous, while in the present paper the nonlinearity is discontinuous. Moreover, since in our paper the functional is not $C^1$, it was necessary to prove a version for Locally Lipschitz Functionals of the linking theorem developed found \cite{G.B Li}, see Sections 4 and 5. Related to the \cite{AG}, we are working with a nonlinearity that is asymptotically linear at infinity, while in that paper the nonlinearity is superlinear.

In order to study the non periodic case, that is, the case where $f$ is not necessarily a periodic function, we will assume that there is 
$h \in C(\mathbb{R}^{N}\times\mathbb{R}, \mathbb{R})$ with $h_{t} \in C(\mathbb{R}^{N}\times\mathbb{R}, \mathbb{R})$ such that 
\begin{itemize}
	\item[(H6)] $h$ is $\mathbb{Z}^{N}$-periodic and 
	\begin{equation*}
		0<h(x,t)t<t^2h_{t}(x,t)\leq V_{\infty}(x)t^2,\;\;\mbox{whenever}\;\;t\neq 0,
	\end{equation*} 
	where $h_{t}$ denotes the derived from function $h$ with respect to $t$ and $V_\infty$ was given in $(H3)$.
	\item [(H7)] $F(x,t)\geq H(x,t)$ for all $t \in \mathbb{R}$ and
	\begin{eqnarray}
		&&	|f(x,t)-h(x,t)|\leq a(x)|t|,\;\forall\;t \in \mathbb{R}, \nonumber \\
		&&|\rho-h(x,t)|\leq  a(x)|t|,\;\forall\;t \in \mathbb{R}\;\;\mbox{and}\;\;\forall\;\rho \in \partial_{t} F(x,t). \nonumber
	\end{eqnarray}
	where $a(x)>0$ for all $x \in \mathbb{R}^N$, $a \in L^{\infty}(\mathbb{R}^{N})$, $a(x)\rightarrow 0$ as $|x|\rightarrow +\infty$ and
	\begin{equation*}
		H(x,t)=\int_{0}^{t}h(x,s)ds.
	\end{equation*}
\end{itemize}
Next, we show an example of a function $f$ that satisfies the assumptions $(H2)-(H7)$. Fixed $b>0$, let us consider the function 
$$
f(x,t)= \left\{\begin{aligned}
	& h(x,t),\;\;\mbox{if}\;\; |t|\leq b \\
	& h(x,t)+\mu \; \gamma\; a(x)\;\arctan(t),\;\; \mbox{if}\;\; |t|> b,
\end{aligned}
\right. 
$$
where  $\mu > \mu_{1}$ and
\begin{itemize}
	\item [(a)]	$h(x,t)=\mu\;[t-\arctan(t)]$;
	\item [(b)] $a(x)=e^{-|x|^{2}}$;
	\item [(c)] $0<\gamma$ is such that $\|a\|_\infty\gamma < \dfrac{\mu_{0}}{\mu}$.
\end{itemize}
\vspace{0.5 cm}

Our main result involving the non periodic case is the following:
\begin{theorem} \label{Teorema2}(The non periodic case)
	Assume $(H1)-(H7)$. Then, the problem $(P)$ has a nontrivial solution.
\end{theorem}

The Theorem \ref{Teorema2} also complements the study made in \cite{G.B Li}, because we are considering that $f$ can be a discontinuous function. 

The plan of the paper is as follows. In Section 2, we recall some definitions and basic results on the critical point theory of Locally Lipschitz functionals. In Section 3, we study a deformation lemma for Locally Lipschitz functionals. In Section 4, we prove a linking theorem for Locally Lipschitz Functionals.  Finally, in Sections 5 and 6, we employ the linking theorem to prove Theorems \ref{Teorema1} and \ref{Teorema2}.

Before concluding this section, still in the context of asymptotically linear problems, we would like to cite the papers \cite{AlamaLi, Maia3, Maia2, Maia, Shaowei,Heerden1, Heerden2, Heerden3, D.G, Jeanjean, Stuart1, Tehrani, Liu.C, Zhou and Zhu, Liu.C1, LiZhou, Qin, Chang.X} 
\vspace{0.5 cm}

\noindent \textbf{Notation:} From now on, otherwise mentioned, we use the following notations:
\begin{itemize}
	
	\item $\|\,\,\,\|_X$ denotes the norm of the space $X$. 
	
	\item $X^{*}$ denotes the dual topological space of $X$ and $||\;\;\;||_{*}$ denotes the norm in $X^{*}$.
	
	\item $B_r(u)$ is an open ball centered at $u \in X$ with radius $r>0$.
	
	
	\item  $||\,\,\,||_p$ denotes the usual norm of the Lebesgue space $L^{p}(\mathbb{R}^N)$, for $p \in [1,+\infty]$.
	
	
	
	
	\item $l:X\rightarrow \mathbb{R}$ denotes a continuous linear functional.
	
	\item $I_{d}:X\rightarrow X$ denotes the identity application.
	
	\item   $C_i$ denote (possibly different) any positive constants, whose values are not relevant.
	
\end{itemize}

\section{Preliminary results}

In this section we recall some facts involving nonsmooth analysis and the energy functional associated with problem $(P)$.

\subsection{Basic results from nonsmooth analysis}

In this subsection, for the reader's convenience, we recall some definitions and basic results on the critical point theory of Locally Lipschitz Functionals as developed by Chang \cite{Chang1}, Clarke \cite{Clarke, Clarke1}, and Grossinho and Tersian \cite{rosario}.

Let $(X, ||\;||_{X})$ be a real Banach space. A functional $I :X \rightarrow \mathbb{R} $ is locally Lipschitz, $I\in Lip_{loc}(X, \mathbb{R})$ for short, if given $u \in X$ there is an open neighborhood $V:=V_{u} \subset X$ of $u$, and a constant $K=K_{u}>0$ such that

\begin{equation*}
|I(v_{2})-I(v_{1})| \leq K ||v_{1}-v_{2}||_{X},\;\;v_{i} \in V,\;i=1,2.
\end{equation*}
The generalized directional derivative of $I$ at $u$ in the direction of $v \in X$ is defined by
\begin{equation*}
I^{\circ}(u;v)= \limsup_{h\rightarrow 0, \delta \downarrow 0}\frac{1}{\delta}\left(I(u+h+\delta v)-I(u+h)\right).
\end{equation*}
The generalized gradient of $I$ at $u$ is the set
\begin{equation*}
\partial I(u)= \{\xi \in X^{*}\;;\; I^{\circ}(u;v) \geq \left<\xi, v\right>\,;\forall\; v \in X \}.
\end{equation*} 
Moreover, we denote by $\lambda_{I}(u)$ the following real number 
$$
\lambda_{I}(u):=\min\{||\xi||_{*}: \xi \in \partial I(u)\}.
$$
We recall that $u \in X$ is a critical point of $I$ if $0 \in \partial I(u)$, or equivalently, when $\lambda_I(u)=0$. 

\begin{lemma}
	If $I$ is continuously differentiable to Fr\'echet in an open neighborhood of $u \in X$, we have $\partial I(u)= \{I'(u)\}$.
\end{lemma}

\begin{lemma}\label{35}
	If $Q \in C^{1}(X, \mathbb{R})$ and $\Psi \in Lip_{loc}(X, \mathbb{R})$, then for each $u \in X$
	$$
	\partial(Q+\Psi)(u)= Q'(u) + \partial \Psi(u).
	$$
\end{lemma}

\subsection{Preliminaries results involving the energy functional}\label{section}
By assumptions $(H2)$ and $(H3)$, the energy functional associated with problem $(P)$ is given by 
$$
I(u)= \frac{1}{2}\int_{\mathbb{R}^{N}}(|\nabla u|^{2}+V(x)u^{2})dx-\int_{\mathbb{R}^{N}}F(x,u)dx, \;u \in H^{1}(\mathbb{R}^{N}),
$$
where $F(x,t)= \int_{0}^{t}f(x,s)ds$.

By standard argument, the functional $Q:H^{1}(\mathbb{R}^{N}) \to \mathbb{R}$ given by  
$$ 
Q(u)= \frac{1}{2}\int_{\mathbb{R}^{N}}(|\nabla u|^{2}+V(x)|u|^{2})dx
$$
belongs to $C^{1}(H^{1}(\mathbb{R}^{N}), \mathbb{R})$ and 
$$
Q'(u)v=\int_{\mathbb{R}^{N}}(\nabla u \nabla v+V(x)u v)dx,\;\forall\; u,v \in H^{1}(\mathbb{R}^{N}).
$$
However, the functional $\Psi: H^{1}(\mathbb{R}^{N})\rightarrow \mathbb{R}$ given by
\begin{equation} \label{Psi}
	\Psi(u)=\int_{\mathbb{R}^{N}}F(x,u)dx
\end{equation}
is only  Locally Lipschitz, because the function $f$ is not continuous in whole $\mathbb{R}$. Hence, the functional $I$ is also Locally Lipschitz, from where it follows that we cannot use the variational methods for $C^{1}$ functionals in order to get critical points for $I$, and so, we must use variational methods for Locally Lipschitz. Have this in mind, by Lemma \ref{35},  
$$
\partial I(u)= Q'(u) - \partial \Psi(u), \quad \forall u \in H^{1}(\mathbb{R}^N).
$$  
By $(H1)$, it is well known that $H^{1}(\mathbb{R}^{N})=E^{+}\oplus E^{-}$ is an orthogonal decomposition and there is an equivalent norm $||\cdot||$ to $||\cdot||_{H^{1}(\mathbb{R}^{N})}$ with
$$
\|u\|^2=||u^{+}||^{2}-||u^{-}||^{2}, \quad \forall\; u=u^{+}+u^{-} \in E^{+} \oplus E^{-}.
$$

Recalling that $(\mu_{-1}, \mu_{1})$ is the spectral gap of $-\Delta+V$ containing $0$ and $\mu_{0}=\min\{-\mu_{-1}, \mu_{1}\}$, by Stuart \cite{Stuart},
	\begin{equation}\label{A83}
	||u^{+}||^{2}\geq \mu_{1} ||u^{+}||_{2}^{2},\;\forall\; u^{+} \in E^{+}\;\;\mbox{and}\;\; ||u^{-}||^{2}\geq -\mu_{-1} ||u^{-}||_{2}^{2},\;\forall\; u^{-} \in E^{-}.
	\end{equation}
	Therefore,
	\begin{equation*}
	||u||^{2}\geq \mu_{0} ||u||_{2}^{2},\;\forall\; u \in H^{1}(\mathbb{R}^{N}).
	\end{equation*}	
Moreover, it is possible to prove that 
\begin{equation} \label{Q}
\langle Q'(u),v\rangle=(u^+,v)-(u^-,v), \quad \forall u,v \in H^{1}(\mathbb{R}^N),
\end{equation}
where $(.,.)$ denotes the usual inner product in $H^{1}(\mathbb{R}^N)$.

Using the notations above, we can rewrite $I$ of the form
\begin{equation}\label{Functional}
I(u)= \frac{1}{2}||u^{+}||^{2}-\frac{1}{2}||u^{-}||^{2}-\Psi(u),\quad \forall\; u=u^{+}+u^{-} \in E^{+} \oplus E^{-}.
\end{equation}

In order to apply variational methods for Locally Lipschitz, we claim that $\Psi: L^{2}(\mathbb{R}^{N})\rightarrow \mathbb{R}$ is well defined, because  by  $(H2)-(H3)$, 
\begin{equation} \label{P1*}
	|f(x,t)|\leq C|t|,\;\forall\;t \in \mathbb{R},\;\forall\;x\in \mathbb{R}^{N},
\end{equation}
and so,
\begin{equation}\label{P1}
	|F(x,t)|\leq C|t|^{2},\;\forall\;t \in \mathbb{R},\;\forall\;x\in \mathbb{R}^{N}.
\end{equation} 
Consequently, 
$$\partial \Psi (u)\subset \partial_{t} F(x,u)= [\underline{f}(x, u(x)), \overline{f}(x, u(x))]\;\;\mbox{a.e in}\;\; \mathbb{R}^{N},$$
where
$$\underline{f}(x,t)= \lim_{r\downarrow 0} ess \inf\{f(x,s); |s-t|<r\}$$
and
$$\overline{f}(x,t)= \lim_{r\downarrow 0} ess \sup\{f(x,s); |s-t|<r\}.$$

The inclusion above means that given $\xi \in  \partial \Psi (u) \subset (L^{2}(\mathbb{R}^{N}))^{*}\approx L^{{2}}(\mathbb{R}^{N})$, there is $\tilde{\xi} \in L^{{2}}(\mathbb{R}^{N})$ such that  
\begin{itemize}
	\item $\left<\xi, v \right>= \int_{\mathbb{R}^{N}} \tilde{\xi}v ,\;\forall\; v \in L^{2}(\mathbb{R}^{N})$,
	\item $\tilde{\xi}(x) \in \partial_{t} F(x, u(x)) = [\underline{f}(x,u(x)), \overline{f}(x,u(x))]\;\;\mbox{a.e in}\;\; \mathbb{R}^{N}.$
\end{itemize}

The following proposition is very important to establish the existence of a critical point for the functional $I$.
\begin{proposition}(See \cite{AG}).\label{7}
	If $(u_{n}) \subset H^{1}(\mathbb{R}^{N})$  is such that 
	$u_{n}\rightharpoonup u_{0}$ in $H^{1}(\mathbb{R}^{N})$ and $\rho_{n} \in \partial \Psi(u_{n})$ satisfies $\rho_{n} \stackrel{\ast}{\rightharpoonup} \rho_{0}$ in $(H^{1}(\mathbb{R}^{N}))^{*}$, then $\rho_{0} \in \partial \Psi(u_{0})$. 
\end{proposition}

\section{A special deformation lemma}	

From now on, $X$ is a Hilbert space with $X=Y \oplus Z$, where $Y$ is a separable closed subspace of $X$ and $Z=Y^{\perp}$. If $u \in X$,  $u^{+}$ and $u^{-}$ denote the orthogonal projections from $X$ in $Z$ and in $Y$, respectively. In $X$ let us define the norm
\begin{eqnarray}
|||\cdot|||: &X &\longrightarrow \mathbb{R} \nonumber \\
&u& \longmapsto |||u|||=\max\left\{||u^{+}||, \sum_{k=1}^{\infty}\frac{1}{2^{k}}|(u^{-},e_{k})|\right\}, \nonumber
\end{eqnarray}
where $(e_{k})$ is a total orthonormal sequence in $Y$. The topology on $X$ generated by $|||\cdot|||$ will be denoted by $\tau$ and all topological notions related to it will include this symbol. From \cite{Kryszewski}, if $(u_{n}) \subset X$ is a bounded sequence, then  
\begin{equation} \label{ob1}
u_{n}\stackrel{\tau}{\rightarrow} u\;\; \mbox{in}\;\; X\Leftrightarrow u_{n}^{-}\rightharpoonup u^{-}\;\; \mbox{and}\;\; u_{n}^{+}\rightarrow u^{+} \quad \mbox{in} \quad X.
\end{equation}
Let $I :X \rightarrow \mathbb{R}$  be a Locally Lipschitz functional, $I\in Lip_{loc}(X, \mathbb{R})$ for short, of the form 
\begin{equation}\label{Func.}
I(u)= \frac{1}{2}||u^{+}||^{2}-\frac{1}{2}||u^{-}||^{2}-\Psi(u),
\end{equation}
where $\Psi \in Lip_{loc}(X, \mathbb{R})$ is bounded from below, $I$ is $\tau$-upper semicontinuous, and $||\cdot||$ is an equivalent norm  to $||\cdot||_{X}$. Moreover, we suppose that
$$\partial I(u)= Q'(u)- \partial \Psi(u),$$
where $Q \in C^{1}(X, \mathbb{R})$ and
$$
\left<Q'(u),v\right>=(u^{+},v)-(u^{-},v),\;u, v \in X,
$$
with $(\cdot, \cdot)$ being the inner product of $X$.

From now on, we will say that a functional $I:X\rightarrow \mathbb{R}$ verifies the condition $(H)$ when: \\
If $(u_{n}) \subset I^{-1}([\alpha, \beta])$ is such that $u_{n}\stackrel{\tau}{\rightarrow} u_{0}$ in $X$, then there exists $M>0$ such that $\partial \Psi(u_{n}) \subset B_{M}(0) \subset X^{*},\;\forall\; n \in \mathbb{N}$. In addition, if  $\xi_{n} \in \partial I(u_{n})$ with $\xi_{n} \stackrel{\ast}{\rightharpoonup} \xi_{0}$ in  $X^{*}$, we have $\xi_{0} \in \partial I(u_{0})$.
\begin{lemma}\label{A45}
If $A \subset X$ and $\varepsilon>0$ is such that
	$$
	(1+||u||)\lambda_{I}(u)\geq \varepsilon,\;\forall u \in A,
	$$
	then, for each $u \in A$,  there exists  $\chi_{u} \in X$ with $||\chi_{u}||=1$ 
	satisfying
	$$
	\left<l,\chi_{u}\right> \geq \frac{\varepsilon}{2}\;,\; \forall\; l \in (1+||u||) \partial I(u).
	$$
\end{lemma}

\begin{proof}
	
	Given $u \in A$ and $l \in \partial I(u)$, we have
$$
	0<\varepsilon \leq  (1+||u||)\lambda_{I}(u) \leq  (1+||u||)||l||_{*}, 
$$
	that is, 
$$
\overline{B}_{\frac{\varepsilon}{2}}(0)\cap (1+||u||)\partial I(u)= \emptyset.
$$
Since $\overline{B}_{\frac{\varepsilon}{2}}(0)$ and $(1+||u||)\partial I(u)$ are closed sets, convex, not empty and disjoint, by Hahn-Banach Theorem there are $\psi_{u} \in X^{**} \backslash \{0\}$ and $R>0$ satisfying 
	$$
	\left<\psi_{u}, l\right> \geq R \geq \left<\psi_{u}, w\right>\;,\;\forall\; l \in (1+||u||)\partial I(u),\forall\; w \in \overline{B}_{\frac{\varepsilon}{2}}(0). 
	$$
	By  reflexivity of $X$, there exists $v_{u} \in X$ such that
	$$
	\left<\psi_{u}, l\right>=\left<l, v_{u}\right>, \;\;\forall\; l \in X^{*}.
	$$
	So, 
	\begin{equation*}
	\left<l, v_{u}\right> \geq \left<w, v_{u}\right>\;,\;\forall\; l \in (1+||u||) \partial I(u),\forall\; w \in\overline{B}_{\frac{\varepsilon}{2}}(0),
	\end{equation*}
	or yet,
	\begin{equation}\label{A43}
	\left<l, \chi_{u}\right> \geq \left<w, \chi_{u}\right>\;,\;\forall\; l \in (1+||u||)\partial I(u),\forall\; w \in \overline{B}_{\frac{\varepsilon}{2}}(0),
	\end{equation}
	where $\chi_{u}=\frac{v_{u}}{||v_{u}||}$.

	By Hahn-Banach Theorem
	$$
	1=||\chi_{u}||=\max\{\left<w, \chi_{u}\right>\;:\;w \in X^* \quad \mbox{and} \quad ||w||_{*}\leq 1\},
	$$
	then,
	$$ 
	\max\left\{\left<\frac{\varepsilon}{2}\; w, \chi_{u}\right>\;:\;w \in X^* \quad \mbox{and} \quad ||w||_{*}\leq 1\right\}= \frac{\varepsilon}{2}. 
	$$ 
	Consequently,
	\begin{equation}\label{A44}
	\max\left\{\left<\tilde{w}, \chi_{u}\right>\;:\;\tilde{w} \in  \overline{B}_{\frac{\varepsilon}{2}}(0)\right\}= \frac{\varepsilon}{2}.
	\end{equation}
	For $ l \in  \partial (1+||u||) I (u)$, (\ref{A43}) and (\ref{A44}) combine to give
	$$
	\left<l, \chi_{u}\right> \geq  \sup_{w \in \overline{B}_{\frac{\varepsilon}{2}}(0)}\left<w, \chi_{u}\right>= 
	\frac{\varepsilon}{2} , 
	$$
	that is,
	$$
	\left<l, \chi_{u}\right> \geq  \frac{\varepsilon}{2},\; \forall\; l \in (1+||u||) \partial I(u),
	$$
	finishing the proof.
\end{proof}

\begin{theorem}\label{A2}
Assume the condition $(H)$ and let $\alpha < \beta $ and $\varepsilon>0$ satisfying 
$$
\lambda_{I}(u)(1+||u||)\geq \varepsilon\;,\; \forall\; u \in I^{-1}([\alpha, \beta]).
$$ 
Then, for each $u_{0} \in I^{-1}([\alpha, \beta])$, there exists $\eta_{0}>0$ such that
$$\left<\xi,\chi_{u_{0}}\right>> \frac{\varepsilon}{3}\;,\; \forall\; \xi \in  (1+||u||)\partial I(u)\;,\; u\in B_{\eta_{0}}(u_{0}) \cap I^{-1}([\alpha, \beta]),$$ 
where $B_{\eta_{0}}(u_{0})=\left\{u \in X\;;\; |||u-u_{0}|||<\eta_{0}\right\}$ with  $\chi_{u_{0}}$ given in Lemma \ref{A45}.
\end{theorem}
\begin{proof}
	
	Arguing by contradiction, assume that there exist $(u_{n}) \subset I^{-1}([\alpha, \beta])$ with $u_{n}\stackrel{\tau}{\rightarrow} u_{0}$ in $X$ and $\xi_{n} \in (1+||u_{n}||)\partial I(u_{n})$ such that
	\begin{equation*}
	\left<\xi_{n}, \chi_{u_{0}}\right> \leq  \frac{\varepsilon}{3},\;\forall\; n \in \mathbb{N}.
	\end{equation*}
	Note that
	$$
	\xi_{n}= (1+||u_{n}||)(Q'(u_{n})- \rho_{n})
	$$
	with $\rho_{n} \in \partial \Psi(u_{n})$. 
	It follows from $(H)$ that, going to a subsequence if necessary, there is $\rho_{0} \in X^{*}$ such that $\rho_{n}\stackrel{\ast}{\rightharpoonup} \rho_{0}$ in $X^{*}$. In addition, $\rho_{0} \in \partial \Psi(u_{0})$. On the other hand, 
	using the lower limitation of $\Psi$, (\ref{Func.}), the limit $u_{n}\stackrel{\tau}{\rightarrow} u_{0}$ in $X$,  
	the fact that $I(u_{n})\geq \alpha \;\forall\;n \in \mathbb{N}$, and  (\ref{ob1}), we can assume that $u_{n}\rightharpoonup u_{0}$ in $X$. Since $\xi_{0}:=(1+||u_{0}||)(Q'(u_{0})- \rho_{0}) \in (1+||u_{0}||)\partial I(u_{0})$, by Lemma \ref{A45},  
	\begin{equation*}
	\frac{\varepsilon}{3} <	\frac{\varepsilon}{2} \leq (1+||u_{0}||)\left<Q'(u_{0})- \rho_{0}, \chi_{u_{0}}\right>.
	\end{equation*}
	Hence, without loss of generality, we can suppose that
	\begin{equation*}
	\left<Q'(u_{n})- \rho_{n}, \chi_{u_{0}}\right> >0,\;\forall\;n \in \mathbb{N}.
	\end{equation*}
	So,
	\begin{eqnarray}
\frac{\varepsilon}{3}<	(1+||u_{0}||)\left<Q'(u_{0})- \rho_{0}, \chi_{u_{0}}\right> &\leq& \liminf_{n \to +\infty}(1+||u_{n}||)\left<Q'(u_{n})- \rho_{n}, \chi_{u_{0}}\right>   \nonumber \\
	&= & \liminf_{n \to +\infty} \left<\xi_{n}, \chi_{u_{0}}\right> \leq \frac{\varepsilon}{3}, \nonumber
	\end{eqnarray}
which is absurd.  
\end{proof}

	\begin{lemma}\label{A3}
	Under the assumptions of Theorem \ref{A2}, there exist a $\tau$-open neighborhood $V$ of $I^{\beta}=\{u \in X\,:\, I(u) \leq \beta\}$ and a vector field $P:V\rightarrow X$ satisfying:
\end{lemma}
\begin{itemize}
	\item [(P1)] $P$ is locally Lipschitz continuous and  $\tau$-locally Lipschitz continuous,
	\item [(P2)] each point $u \in V$ has a $\tau$-neighborhood $V_{u}$ such that $P(V_{u})$ is contained in a finite-dimensional subspace of $X$,
		\item [(P3)] for $u \in V$,  $ ||P(u)|| \leq  1+2||u||$ and $\left<\xi,P(u)\right>\geq 0,\;\forall\; \xi\in (1+||u||)\partial I(u)$;
	\item [(P4)] for $u \in I^{-1}([\alpha, \beta])$
	$$\left<\xi,P(u)\right>\geq \frac{\varepsilon}{3},\;\forall\; \xi \in (1+||u||)\partial I(u).$$
\end{itemize}

\begin{proof}
	
	For each $u_{0} \in I^{-1}([\alpha, \beta])$, the Theorem \ref{A2} guarantees the existence of $\eta_{0}>0$ satisfying
	\begin{equation}\label{A5}
	\left<\xi, \chi_{u_{0}}\right> \geq \frac{\varepsilon}{3},\;\forall\; \xi \in (1+||u||)\partial I(u),\;\;u \in B_{\eta_{0}}(u_{0}) \cap I^{-1}([\alpha, \beta]).
	\end{equation}
	\begin{claim}\label{A4}
	For each $u \in I^{-1}([\alpha, \beta])$ there exists a $\tau$-neighborhood $N_{u}$ of $u$ such that
		\begin{equation}\label{Sent.}
		||u|| \leq  2||v||,\;\forall\; v \in N_{u}.
		\end{equation}
	\end{claim}
If $u=0$, then (\ref{Sent.}) 
is immediate. Now, for $u \in I^{-1}([\alpha, \beta])\backslash \{0\}$, suppose that (\ref{Sent.}) is not true. Then, there exists 
$$
(v_{j}) \subset B_{\frac{1}{j}}(u)=\left\{v \in X: |||v-u|||\leq \frac{1}{j}\right\}
$$ 
with
	\begin{equation*}
	2 ||v_{j}|| \leq ||u||,\;\forall\; j \in \mathbb{N}.
	\end{equation*}
As $v_{j}	\stackrel{\tau}{\rightarrow} u$ in $X$, by (\ref{ob1}), we derive that $v_{j}\rightharpoonup u$ in $X$. Consequently, 
	$$
	0<||u|| \leq \liminf_{j}||v_{j}|| \leq \frac{||u||}{2},
	$$
	which is absurd. 
	
	Consider $U_{u}= B_{\eta_{u}}(u)\cap I^{-1}([\alpha, \beta])\cap N_{u}$,  which is still a $\tau$-neighborhood of $u \in I^{-1}([\alpha, \beta])$. Since $I$ is $\tau$-upper semicontinuous, then 
	$$
	U_{0}=I^{-1}((-\infty, \alpha))
	$$ is $\tau$-open in $X$. Therefore, the family 
	$$
	\mathcal{N}:=\{U_{u}\}_{u \in I^{-1}([\alpha, \beta])}\cup U_{0}
	$$
	is a $\tau$-open covering for $I^{-1}((-\infty, \beta]):=I^{\beta}$.
	
	 In addition $(I^{\beta}, \tau)$ is a metric space, then there exists a  $\tau$-locally finite $\tau$-open covering  $\mathcal{V}=\left\{\mathcal{V}_{i}\;:\; i \in \mathcal{J}\right\}$ of $I^{\beta}$ (see \cite{Ryszard}) more fine than $\mathcal{N}$. Next, we define the $\tau$-open neighborhood of $I^{\beta}$ by 
	$$
	V= \bigcup_{i \in \mathcal{J}} \mathcal{V}_{i}
	$$
	and set $\left\{\gamma_{i}\;:\; i \in \mathcal{J}\right\}$ as being  a $\tau$-Lipschitz continuous partition of unity subordinated to $\mathcal{M}$. Employing the notations above, we set the vector field $P:V \to X$ by 
	$$
	P(u)=\sum_{i \in \mathcal{J}} \gamma_{i}(u)w_{i}, 
	$$
	where:
	\begin{itemize}
		\item If $\mathcal{V}_{i} \subseteq U_{u_i}$, we choose $w_{i}=\chi_{u_{i}}(1+||u_{i}||)$ ( where $\chi_{u_{i}} \in X$ is given in Lemma \ref{A45}).
		\item If $\mathcal{V}_{i}\subseteq U_{0}$, we choose $w_{i}=0$. 
	\end{itemize}
	The items $(P1)$, $(P2)$ and $(P4)$ follow in a similar way as done in \cite{AG}. 
\begin{itemize}
	\item [(P3)] Note that, if $u\in U_{u_{i}}$, by Claim \ref{A4}, 
	\begin{eqnarray}
	||w_{i}||=||\chi_{u_{i}}(1+||u_{i}||)||&=& (1+||u_{i}||)\nonumber\\
	&\leq& (1+2||u||). \nonumber
	\end{eqnarray}
	Therefore, given $u \in V$, 
	\begin{eqnarray}
	||P(u)|| &\leq & \sum_{i \in \mathcal{J}} \gamma_{i}(u) ||w_{i}|| \nonumber \\
	&\leq& (1+2||u||) \sum_{i \in \mathcal{J}} \gamma_{i}(u) =(1+2||u||). \nonumber
	\end{eqnarray}	
Let $u \in V$, then $u \in \mathcal{V}_{i}$ for some $i \in \mathcal{J}$.
	\begin{itemize}
		\item If $\mathcal{V}_{i} \subset U_{0}$, then
		$$(\xi, P(u))=0,\;\forall\; \xi \in (1+||u||)\partial I(u).$$
		\item If $\mathcal{V}_{i} \subset U_{u_{i}}$, for some $i \in \mathcal{J}$, by (\ref{A5}), for $\xi \in (1+||u||)\partial I(u)$, we have
		\begin{eqnarray}
		(\xi, P(u)) &=& \sum_{i \in \mathcal{J}}  \gamma_{i}(u)\left(\xi, \chi_{u_{i}}(1+||u_{i}||) \right) \nonumber \\
		&= & \sum_{i \in \mathcal{J}} (1+||u_{i}||) \gamma_{i}(u)\left(\xi, \chi_{u_{i}} \right) \nonumber \\
		&\geq & \frac{\varepsilon}{3}\sum_{i \in \mathcal{J}}  \gamma_{i}(u)= \frac{\varepsilon}{3} > 0. \nonumber
		\end{eqnarray}
	\end{itemize}
\end{itemize}
\end{proof}
Finally we are ready to state our deformation lemma, whose proof follows as in \cite{AG}.
\begin{lemma}\label{A10}(Deformation lemma)
	Under the assumptions of Lemma \ref{A3}, there exists a vector field $\eta: \mathbb{R}^{+}\times I^{\beta}\rightarrow X$ that satisfies the following properties:
\end{lemma}
\begin{itemize}
	\item [(a)] There exists $T>0$ such that 
	$$\eta(T, I^{\beta}) \subset I^{\alpha};$$
	\item [(b)] Each point $(t,u) \in [0,T] \times I^{\beta}$ has a $\tau$-neighborhood $N_{(t,u)}$ such that
	$$\left\{v-\eta (s,v) \in N_{(t,u)}\cap  ([0,T] \times I ^{\beta}) \right\} $$
	is contained in a finite-dimensional subspace of $X$;
	\item [(c)] $\eta$ is continuous in  $[0,+\infty) \times (X,\tau)$ endowed with the norm $\|(t,u)\|_{\bigstar}=|t|+|||u|||$.
\end{itemize}

\section{Generalized linking theorem}
Let $Y$ be a separable closed subspace of a Hilbert space $X$, $Z=Y^{\perp}$ and $I:X\rightarrow \mathbb{R}$ is a locally Lipschitz functional . If $u \in X$, as in the previous section, $u^{+}$ and $u^{-}$ denote the orthogonal projections in $Z$ and $Y$, respectively. 

We say that $(u_{n})\subset X$ is a Cerami sequence at the level $c$ for $I$ ($(C)_{c}$-sequence for short), if
\begin{equation*}
I(u_{n})\rightarrow c\;\;\mbox{and}\;\; (1+||u_{n}||)\lambda_{I}(u_{n})\rightarrow 0.
\end{equation*}
	Let $U$ be an open subset of $X$. 
A homotopy $\gamma=I_{d}-g: [0,1] \times \overline{U} \rightarrow X$ is said to be admissible if:
\begin{itemize}
	\item [(a)] $\gamma$ is $\tau$-continuous, i.e., $\gamma(s_{n},u_{n}) \stackrel{\tau}{\rightarrow} \gamma(s,u)$ whenever $u_{n}\stackrel{\tau}{\rightarrow} u$ in $X$ and $s_{n}\rightarrow s$ in $\mathbb{R}$;
	\item [(b)] $g$ is $\tau$-locally finite-dimensional, i.e., for each  $(t,u) \in [0,1] \times U$ there is a  neighbourhood  $\mathcal{N}$ of $(t,u)$ in the product topology of $[0,1]$ and $(X, \tau)$ such that $g(\mathcal{N} \cap ([0,1]\times U))$ is contained in a finite-dimensional subspace of $X$.
\end{itemize}
Given $R >r>0$ and $z \in Z\backslash\{0\}$, we set
\begin{eqnarray}
&&\mathcal{M}=\left\{u=y+t z\;; ||u||\leq R, t \geq 0 \;\;\mbox{and}\;\; y \in Y\right\} \nonumber \\
&& \mathcal{M}_{0}=\left\{u=y+t z\;;y \in Y, ||u||= R\;\;\mbox{and}\;\; t \geq 0 \;\;\mbox{or}\;\;  ||u||\leq  R\;\;\mbox{and}\;\; t=0 \right\} \nonumber \\
&& S=\left\{u \in Z\;; ||u||=r\right\} \nonumber
\end{eqnarray} 
and
\begin{eqnarray}
 &\Gamma:=\{\gamma\in C([0,1] \times \mathcal{M}, X): \gamma \;\mbox{is admissible},\; \gamma(0,u)=u\nonumber \\ 
&\mbox{and}\; I(\gamma(s,u)) \leq \max\{I(u), -1 \},\forall\; s \in [0,1] \}. \nonumber
\end{eqnarray}
\begin{theorem}(Linking theorem).\label{A85}
Let $X=Y \oplus Z$ be a separable Hilbert space with $Y$ orthogonal to $Z$. Suppose that:
	\begin{itemize}
		\item [(i)]  $I \in Lip_{loc}(X, \mathbb{R})$ is $\tau$-upper semicontinuous with
		$$
		I(u)=\frac{1}{2}||u^{+}||^{2}-\frac{1}{2}||u^{-}||^{2}-\Psi(u),
		$$
		where  $\Psi \in Lip_{loc}(X, \mathbb{R})$ is bounded from below.
		\item [(ii)] There exists $z_{0} \in Z\backslash \{0\}$, $\delta>0$ and $R>r>0$ such that $I\lvert_{S} \geq \delta$ and $I\lvert_{\mathcal{M}_{0}} \leq 0$.
	\end{itemize}
	If $I$ satisfies the condition $(H)$, then there exists a sequence $(C)_{c}$ for $I$, where 
	$$c=\inf_{\gamma \in \Gamma} \sup_{u \in \mathcal{M}}I(\gamma(1,u)).$$
	In addition, $c\geq \delta$.
\end{theorem}
\begin{proof}

The inequality $c\geq \delta$ follows  as in 	\cite{G.B Li}. Suppose by contradiction that there is $\varepsilon>0$ such that
\begin{equation*}
(1+||u||)||\lambda_{I}(u)|| \geq \varepsilon,\;\forall\; u \in I^{-1}([c-\varepsilon, c+\varepsilon]).
\end{equation*}
As $I$ verifies the condition $(H)$, we can employ the  Lemma \ref{A10} with $\alpha=c-\varepsilon$ and $\beta=c+\varepsilon$. 

Since 
$$
c= \inf_{\gamma \in \Gamma} \sup_{u \in \mathcal{M}} I(\gamma(1,u)),
$$
there is $\tilde{\gamma} \in \Gamma$ such that $\tilde{\gamma}(\{1\}\times \mathcal{M}) \subset I^{c+\varepsilon}$. Thereby, by Lemma \ref{A10}, there exists $T>0$ such that
\begin{equation}\label{A12}
I(\eta(T, \tilde{\gamma}(1,u))) \leq c-\varepsilon, \quad \forall u \in I^{c+\varepsilon}.
\end{equation}
Moreover, according to Lemma \ref{A10}, $\eta:[0, T]\times I^{c+\varepsilon}\rightarrow X$ is an admissible homotopy.

Now, let us consider the homotopy $\gamma:[0,1]\times \mathcal{M}\rightarrow X$ defined as follows:
\begin{eqnarray} 
\gamma(s,u)=\left\{\begin{array}{c}
\tilde{\gamma}(2s,u),\;\;\;\;\;\;\;\;\;0\leq s\leq \frac{1}{2}  \nonumber\\
\eta\left( T(2s-1), \tilde{\gamma}(1,u)\right),\;\; \frac{1}{2} \leq s \leq 1.
\end{array}
\right.
\end{eqnarray}
Then $\gamma \in \Gamma$, and by (\ref{A12}),
$$
\sup_{u \in \mathcal{M}} I(\gamma(1,u)) \leq c-\varepsilon,
$$ 
which contradicts the definition of $c$.
\end{proof}
	\begin{corollary}\label{N4}
	Under the hypotheses of Theorem \ref{A85} and assuming that
	\begin{equation*}
	c=\sup_{u \in \mathcal{M}} I(u),
	\end{equation*}
	there is $v \in \mathcal{M}$ such that $I(v)=c$ and $0 \in \partial I(v)$.
\end{corollary}
\begin{proof}
Seeking for a contradiction, we suppose that $\mathcal{M} \cap K_{c}=\emptyset$ where
	\begin{equation*}
	K_{c}=\{u \in X: I(u)=c\;\;\mbox{and}\;\;0 \in \partial I(u)\}.
	\end{equation*}
	\begin{claim}\label{N2}
		There exists $\varepsilon>0$ such that
		\begin{equation*}
		\lambda_{I}(u)\geq \varepsilon,\;\forall\; u \in I^{-1}([c-\varepsilon, c+\varepsilon])\cap \mathcal{M}.
		\end{equation*}
	\end{claim}
	In fact, otherwise there is $(u_{n})\subset \mathcal{M}$ such that
	\begin{equation*}
	\lambda_{I}(u_{n})\rightarrow 0\;\;\mbox{and}\;\; I(u_{n})\rightarrow c.
	\end{equation*}
As $||u_{n}||\leq R$ and $X$ is a Hilbert space, going to a subsequence if necessary, there is $u \in X$ such that
	\begin{equation*}
	u_{n}\rightharpoonup u \;\;\mbox{in}\;\;X.
	\end{equation*}
	Hence, $||u||\leq \displaystyle \liminf_{n \to +\infty}||u_{n}||\leq R$.
	
	On the other hand, using the fact that $u_{n}=u_{n}^{-}+t_{n}u_{0}^{+}$, it follows that $(t_n) \subset [0, +\infty)$ and $(u_{n}^{-})\subset X^{-}$ are bounded sequences. Thus,  going to a subsequence if necessary, there are $u^{-} \in X^{-}$ and $t_{0}\in [0, +\infty)$ such that
	\begin{equation}\label{N10}
	u_{n}^{-}\rightharpoonup u^{-}\;\;\mbox{and}\;\; t_{n}\rightarrow t_{0},
	\end{equation}
	then  $u_{n}\rightharpoonup u^{-}+t_{0}u_{0}^{+}$, and so, $u=u^{-}+t_{0}u_{0}^{+}$ and $||u||\leq R$, that is, $u \in \mathcal{M}$. 
	In addition, using (\ref{N10}), we deduce that $c \leq I(u)$,  because by Fatou's lemma
	\begin{eqnarray}
	c&=&\limsup_{n \to +\infty} I(u_{n}) \nonumber\\
	&=& \limsup_{n \to +\infty}\left(\frac{t_{n}^{2}}{2}||u_{0}^{+}||^{2}- \frac{1}{2}||u_{n}^{-}||^{2}-\Psi(u_{n})\right) \nonumber \\
	&=& \frac{t_{0}^{2}}{2}||u_{0}^{+}||^{2}- \frac{1}{2}\liminf_{n} ||u_{n}^{-}||^{2}- \liminf_{n} \Psi(u_{n}) \nonumber \\
	&\leq  & \frac{t_{0}^{2}}{2}||u_{0}^{+}||^{2}- \frac{1}{2}||u^{-}||^{2}- \Psi(u)=I(u).\nonumber
	\end{eqnarray}
	Finally, since $I(u)\leq \displaystyle\sup_{\mathcal{M}} I=c$, we can conclude that $I(u)=c$.
	
	Now we are ready to show that $0 \in \partial I(u)$. 
	First of all note that $u_{n}\stackrel{\tau}{\rightarrow} u$ in $X$, because $(u_{n})\subset X$ is bounded,  
	$$
	u_{n}^{+}\rightarrow u^{+}\;\;\mbox{and}\;\; u_{n}^{-}\rightharpoonup u^{-} \quad \mbox{in} \quad X.
	$$
	In what follows, let us set $w_{n} \in \partial I(u_{n})$ with $\lambda_{I}(u_{n})=||w_{n}||_{*}$ and $\rho_{n} \in \partial \Psi(u_{n})$ satisfying 
	\begin{equation*}
	\left<w_{n}, \varphi \right>= \left<Q'(u_{n}), \varphi \right>- \left<\rho_{n}, \varphi \right>,\;\forall\;\varphi \in  X \quad \mbox{and} \quad \forall\; n \in \mathbb{N}.
	\end{equation*}
Since $	\lambda_{I}(u_{n})\rightarrow 0$ and $u_{n}\rightharpoonup u$ in $X$, 
	\begin{equation*}
	\left<\rho_{n}, \varphi \right>\rightarrow  \left<Q'(u), \varphi \right> ,\;\forall\;\varphi \in  X,
	\end{equation*}
	that is, $\rho_{n} \stackrel{\ast}{\rightharpoonup}Q'(u)$ in $X^{*}$. Therefore, $Q'(u) \in \Psi(u)$ (condition (H)), that is, $Q'(u)=\rho$ for some $\rho \in \Psi(u)$, then   $0 \in \partial I(u)$. Therefore, $ u \in K_{c}\cap \mathcal{M}$, which is absurd because we are supposing that $K_{c}\cap \mathcal{M} = \emptyset$. This proves the Claim \ref{N2}.
	
	By hypothesis, $I(u)\leq c$ for all $u \in \mathcal{M}$. Thus, Claim \ref{N2} together with Lemma \ref{A10} yield there is $T>0$ such that 
	\begin{equation}\label{N3}
	I(\eta(T, u))\leq c-\varepsilon, \quad \forall u \in \mathcal{M}. 
	\end{equation}
	Now, let us consider the homotopy $\overline{\gamma}:[0,1]\times \mathcal{M}\rightarrow X$ defined by 
	\begin{eqnarray} 
	\overline{\gamma}(s,u)=\left\{\begin{array}{c}
	u,\;\;\;\;\;\;\;\;\;0\leq s\leq \frac{1}{2}  \nonumber\\
	\eta\left( T(2s-1), u\right),\;\; \frac{1}{2} \leq s \leq 1.
	\end{array}
	\right.
	\end{eqnarray}
	Analogous to what was done in Theorem \ref{A85}, we have $\overline{\gamma} \in \Gamma$. In addition, (\ref{N3}) 
	ensures that $\overline{\gamma}(\{1\} \times \mathcal{M}) \subset I^{c-\varepsilon}$, that is,
	$$\sup_{u \in \mathcal{M}} I(\overline{\gamma}(1,u)) \leq c-\varepsilon,$$  
	which contradicts the definition of $c$.
\end{proof}

\section{Proof of Theorem \ref{Teorema1}}

In order to prove Theorem \ref{Teorema1}, we would like point out that the same arguments explored in \cite{AG} guarantee that the energy function $I$ associated with problem $(P)$, see (\ref{Functional}), satisfies the condition $(H)$.

In what follows, our goal is to show that functional $I$ verifies the link geometry of the Theorem \ref{A85} with $X=H^{1}(\mathbb{R}^N)$, $Z=E^+$ and $Y=E^-$.
\begin{lemma}\label{A25}
	Suppose $(H2)-(H3)$. Then, there are $\beta>0$ and $r>0$ such that
	$$\inf_{u \in S}I(u) \geq \beta.$$
\end{lemma}
\begin{proof}
	
	From (\ref{A83}), for $0<\varepsilon<\frac{\mu_{1}}{2}$, the Sobolev continuous embedding together with (\ref{P1}) leads to 
	\begin{eqnarray}
	I(u)&=&\frac{1}{2}||u||^{2}- \int_{\mathbb{R}^{N}}F(x,u)\;dx \nonumber\\
	&\geq& \frac{1}{2}||u||^{2}- \frac{\varepsilon}{\mu_{1}}||u||^{2}- \tilde{C_{\varepsilon}}||u||^{p} \nonumber \\
	&=& \left(\frac{\mu_{1}-2\varepsilon}{\mu_{1}}\right)\frac{||u||^{2}}{2}- \tilde{C_{\varepsilon}}||u||^{p}. \nonumber
	\end{eqnarray}
	From this, there are $\beta, r> 0$ of a such way  that
	\begin{equation*}
	I(u)\geq\beta>0,\; \mbox{for} \quad ||u||=r.
	\end{equation*}
\end{proof}

\begin{lemma}\label{A22}
	Suppose $(H2)-(H3)$. If $z_{0} \in E^{+}\backslash \{0\}$ is such that
	\begin{equation}\label{A52}
	\tau^{2} ||z_{0}||^{2}-||y||^{2}- \int_{\mathbb{R}^{N}}V_{\infty}(x)(\tau z_{0}+y)^{2}\;dx<0, \quad \forall \tau>0 \quad \mbox{and} \quad y \in X^{-}, 
	\end{equation}
	 then
	\begin{equation*}
	\sup_{u \in \mathcal{M}_{0}}I(u) \leq 0.
	\end{equation*}
\end{lemma}
\begin{proof}
	
Let $E=E^{-}\oplus \mathbb{R} z_{0}\equiv E^{-}\oplus \mathbb{R}_{+} z_{0}$. If $u \in E$, then
	$$I(u)= \frac{t^{2}}{2}||z_{0}||^{2}- \frac{||y||^{2}}{2}-\int_{\mathbb{R}^{N}} F(x,y+tz_{0})\;dx.$$
	\begin{claim}\label{A18}
		There exists $R>0$ such that
		\begin{equation*}
		I(u) \leq 0\;\;\mbox{for}\;\;||u||=R.
		\end{equation*}
	\end{claim}
Indeed, suppose that there are $(y_{n}) \subset E^{-}$ and $(t_{n}) \subset [0,+\infty)$ such that $||y_{n}+t_{n} z||\rightarrow +\infty$ and $I(y_{n}+t_{n}z_{0})> 0$ for all $n \in \mathbb{N}$. So, for $u_{n}=y_{n}+t_{n} z$, 
	\begin{eqnarray}
	0<\frac{I(u_{n})}{||u_{n}||^{2}}&= & \frac{1}{2}\frac{t_{n}^{2}}{||u_{n}||^{2}}||z_{0}||^{2}- \frac{1}{2}\frac{||y_{n}||^{2}}{||u_{n}||^{2}}-\int_{\mathbb{R}^{N}} \frac{F(x,u_{n})}{||u_{n}||^{2}}\;dx, \nonumber 
	\end{eqnarray}
	that is,
	\begin{equation}\label{A15}
	\frac{I(u_{n})}{||u_{n}||^{2}}=  \frac{1}{2}\tau_{n}^{2}||z_{0}||^{2}- \frac{1}{2}||v_{n}||^{2}-\int_{\mathbb{R}^{N}} \frac{F(x,u_{n})}{||u_{n}||^{2}}\;dx
	\end{equation}
	where
	$$\tau_{n}=\frac{t_{n}}{||u_{n}||} \;\;\mbox{and}\;\;v_{n}=\frac{y_{n}}{||u_{n}||}.$$
	Note that 
$$
	\tau_{n}^{2}||z_{0}||^{2}+||v_{n}||^{2}= \frac{||t_{n}z_{0}||^{2}+||y_{n}||^{2}}{||u_{n}||^{2}} = \frac{||u_{n}||^{2}}{||u_{n}||^{2}}=1. \nonumber
$$
Hence, going to a subsequence if necessary, there are $\tau \geq 0$ and $v \in X^{-}$ such that
	\begin{equation*}
	\tau_{n}\rightarrow \tau\;\;\mbox{in}\;\;\mathbb{R}\;\;\mbox{and}\;\; v_{n}\rightharpoonup v\;\;\mbox{in}\;\;H^{1}(\mathbb{R}^{N}).
	\end{equation*}
	Setting
	$$
	\overline{v}=\tau z_{0}+v\;\;\mbox{and}\;\; \overline{v}_{n}= \tau_{n} z_{0}+v_{n},
	$$
	we obtain
	\begin{equation*}
	\overline{v}_{n}= \frac{u_{n}}{||u_{n}||} \quad \mbox{and}\;\;	\overline{v}_{n}\rightharpoonup \overline{v}\;\;\mbox{in}\;\;H^{1}(\mathbb{R}^{N}). 
	\end{equation*}	
	For each $R>0$, (\ref{A15}) combined with $(H3)$ gives 
	\begin{eqnarray}
	0 \leq  \frac{I(u_{n})}{||u_{n}||^{2}}&=&  \frac{1}{2}\tau_{n}^{2}||z_{0}||^{2}- \frac{1}{2}||v_{n}||^{2}-\int_{\mathbb{R}^{N}} \frac{F(x,u_{n})}{||u_{n}||^{2}}\;dx \nonumber \\
	&\leq&   \frac{1}{2}\tau_{n}^{2}||z_{0}||^{2}- \frac{1}{2}||v_{n}||^{2}-\frac{1}{2}\int_{B_{R}} \frac{V_{\infty}(x)u_{n}^{2}}{||u_{n}||^{2}}\;dx-\int_{B_{R}}\frac{F_{\infty}(x,u_{n})}{||u_{n}||^{2}}\;dx \nonumber \\
	&=& \frac{1}{2}\tau_{n}^{2}||z_{0}||^{2}- \frac{1}{2}||v_{n}||^{2}-\frac{1}{2}\int_{B_{R}} V_{\infty}(x) \overline{v}_{n}^{2}\;dx-\int_{B_{R}}\frac{F_{\infty}(x,u_{n})}{||u_{n}||^{2}}\;dx. \nonumber
	\end{eqnarray}
In what follows, we will show that 
$$
\lim_{n \to +\infty}\int_{B_{R}}\frac{F_{\infty}(x,u_{n})}{||u_{n}||^{2}}\;dx=0.
$$
Since
$$
 \int_{B_{R}}\frac{F_{\infty}(x,u_{n})}{||u_{n}||^{2}}\;dx= \int_{B_{r}}\frac{F_{\infty}(x,u_{n})}{u_{n}^{2}}\frac{u_{n}^{2}}{||u_{n}||^{2}}\;dx=\int_{B_{R}}\frac{F_{\infty}(x,u_{n})}{u_{n}^{2}}\overline{v}_{n}^{2}\;dx
 $$
it is enough to show that 
$$
\lim_{n \to +\infty}\int_{B_{R}}\frac{F_{\infty}(x,u_{n})}{u_{n}^{2}}\overline{v}_{n}^{2}\;dx=0.
$$
In the proof this limit, the condition $(H3)$ applies an essential rule, because it ensures that 
\begin{equation*}
	\lim_{|t|\rightarrow +\infty}\frac{F_{\infty}(x,t)}{t^{2}}=0
\end{equation*}
and  
$$
|F_{\infty}(x,t)| \leq \frac{c_{0}}{2}|t|^{2}, \quad \forall t \in \mathbb{R} \quad \mbox{and} \quad \forall x \in \mathbb{R}^N,  
$$
for some $c_0>0$

Next, we will analyze the cases $\overline{v} = 0$ and $\overline{v}\not=0$. 
	\begin{itemize}
		\item If $\overline{v}=0$, then
		\begin{eqnarray}
		\left|\int_{B_{R}}\frac{F_{\infty}(x,u_{n})}{u_{n}^{2}}\overline{v}_{n}^{2}\;dx\right| &\leq & \frac{c_{0}}{2} \int_{B_{R}} |\overline{v}_{n}|^{2}\;dx. 
		\end{eqnarray}
	Using the fact that $\overline{v}_{n}\rightarrow 0$ in $L^{2}(B_{R})$, we derive 
		\begin{equation*}
		\lim_{n\rightarrow +\infty} \int_{B_{R}}\frac{F_{\infty}(x,u_{n})}{u_{n}^{2}}\overline{v}_{n}^{2}\;dx=0.
		\end{equation*}
		\item If $\overline{v}\neq 0$, as $\overline{v}_{n}(x)\rightarrow \overline{v}(x)$ a.e in $B_{R}$ and $||u_{n}||\rightarrow +\infty$, we must have $u_{n}(x)=\overline{v}_{n}(x)||u_{n}||\rightarrow +\infty$ a.e. in $B_R(0)$, and so, 
		\begin{equation*}
		\lim_{n\rightarrow +\infty}\frac{F_{\infty}(x,u_{n})}{u_{n}^{2}}\overline{v}_{n}^{2}= 0,\;\mbox{a.e in}\;\;B_{R}.
		\end{equation*}
	\end{itemize}
As
	$$
	\left|\frac{F_{\infty}(x,u_{n})}{u_{n}^{2}}\overline{v}_{n}\right| \leq \frac{c_{0}}{2}|\overline{v}_{n}|^{2},
	$$
the Lebesgue's dominated convergence theorem yields
	$$
	\lim_{n \to +\infty}\int_{B_{R}}\frac{F_{\infty}(x,u_{n})}{u_{n}^{2}}\overline{v}_{n}^{2}\;dx=0,
	$$
	or equivalently,
	$$
	\lim_{n \to +\infty}\int_{B_{R}}\frac{F_{\infty}(x,u_{n})}{||u_{n}||^{2}}\;dx=0.
	$$
	Therefore,
	\begin{eqnarray}
	0&\leq&\limsup_{n \to +\infty}\left(\frac{1}{2}\tau_{n}^{2}||z_{0}||^{2}- \frac{1}{2}||v_{n}||^{2}-\frac{1}{2}\int_{B_{R}} V_{\infty}(x) \overline{v}_{n}^{2}\;dx-\int_{B_{R}}\frac{F_{\infty}(x,u_{n})}{||u_{n}||^{2}}\;dx\right) \nonumber \\
	&\leq & \frac{1}{2}\tau^{2}||z_{0}||^{2}- \frac{1}{2}\liminf_{n \to +\infty}||v_{n}||^{2}-\frac{1}{2}\int_{B_{R}} V_{\infty}(x) \overline{v}^{2}\;dx \nonumber \\
	&\leq & \frac{1}{2}\tau^{2}||z_{0}||^{2}- \frac{1}{2}||v||^{2}-\frac{1}{2}\int_{B_{R}} V_{\infty}(x) (\tau z_{0}+v)^{2}\;dx. \nonumber 
	\end{eqnarray}
	Taking  the limit of $R \to +\infty$ and using the fact that $V_{\infty}(\tau z_{0}+y)^{2} \in L^{1}(\mathbb{R}^{N})$, we find the inequality below  
	$$
	0\leq \frac{1}{2}\tau^{2}||z_{0}||^{2}- \frac{1}{2}||v||^{2}-\frac{1}{2}\int_{\mathbb{R}^{N}} V_{\infty}(x) (\tau z_{0}+v)^{2}\;dx,
	$$
	which contradicts (\ref{A52}). This finishes the proof of Claim \ref{A18}, and the lemma is proved.
\end{proof}
Before continuing our study, we would like to point out that the condition (\ref{A52}) is not empty, because by \cite[Remark 3.6]{G.B Li} there exists $z_{0} \in X^{+}$ with $||z_{0}||=1$, such that
	\begin{equation*}
	t^{2} ||z_{0}||^{2}-||y||^{2}- \mu||y+t z_{0}||_{2}^{2}<0, \quad \forall t \in \mathbb{R} \quad \mbox{and} \quad \forall y \in E^{-}. 
	\end{equation*}
Hence, by $(H3)$, 
	$$
	\tau^{2} ||u_{0}^{+}||^{2}-||y||^{2}- \int_{\mathbb{R}^{N}} V_{\infty}(x)(\tau u_{0}^{+}+y)^{2}dx<0, \quad \forall t \in \mathbb{R} \quad \mbox{and} \quad \forall y \in E^{-}.
	$$

\begin{lemma}\label{A26}
All sequences $(C)_{c}$ for the functional $I$ are bounded.
\end{lemma}
\begin{proof}
	
	Let $(u_{n}) \subset H^{1}(\mathbb{R}^{n})$ be a sequence $(C)_{c}$ for the functional $I$, that is,
	\begin{equation*}
	I(u_{n})\rightarrow c\;\;\mbox{and}\;\; (1+||u_{n}||)\lambda_{I}(u_{n})\rightarrow 0\;\; \mbox{as}\;\;n\rightarrow +\infty.
	\end{equation*}	
	Next, we set $w_{n} \in \partial I(u_{n})$ with $\lambda_{I}(u_{n})=||w_{n}||_{*}$ and $\rho_{n} \in \partial \Psi(u_{n})$ such that 
	\begin{equation*}
	w_{n}=Q'(u_{n})-\rho_{n}.
	\end{equation*}

	Suppose by contradiction that  for some subsquence, still denoted by $(u_n)$, 
	\begin{equation}
	||u_{n}||\rightarrow +\infty\;\;\mbox{as}\;\;n\rightarrow +\infty.
	\end{equation}
	Setting the sequence $v_{n}=\frac{u_{n}}{||u_{n}||}$, then $(v_{n})$ is either: 
	\begin{itemize}
		\item [(1)] (\it Vanishing): For each $r>0$ 
		$$
		\lim_{n \to +\infty}\sup_{y \in \mathbb{R}^{N}} \int_{B_{r}(y)}|v_{n}|^{2}\;dx=0,
		$$
		or
		\item [(2)] (\it Non-vanishing): There exist $r, \eta>0$ and a sequence $(z_{n}) \subset \mathbb{Z}^{N}$ such that
		$$
		\limsup_{n \to +\infty}\int_{B_{r}(z_{n})}|v_{n}|^{2}\;dx\geq \eta.
		$$
	\end{itemize}
	Suppose first that $(v_{n})$ is non-vanishing. 	Given $\varphi \in C_{0}^{\infty}(\mathbb{R}^{N})$ and $\varphi_{n}(x)=\varphi(x-z_{n})$, we obtain that 
	\begin{eqnarray}
	o_n(1)=\frac{1}{||u_{n}||}\left<w_{n}, \varphi_{n}\right>&=& \left<v_{n}^{+}-v_{n}^{-}, \varphi_{n}\right>-\frac{1}{||u_{n}||} \int_{\mathbb{R}^{N}}\rho_{n}\varphi_{n}\;dx \nonumber \\
	&=& \left<v_{n}^{+}-v_{n}^{-}, \varphi_{n}\right>-\frac{1}{||u_{n}||} \int_{\mathbb{R}^{N}}V_{\infty}(x)u_{n}\varphi_{n}\;dx-\frac{1}{||u_{n}||} \int_{\mathbb{R}^{N}} \rho_{n}^{\infty} \varphi_{n}\;dx \nonumber\\
	&=& \left<v_{n}^{+}-v_{n}^{-}, \varphi_{n}\right>- \int_{\mathbb{R}^{N}}V_{\infty}(x)v_{n}\varphi_{n}\;dx- \int_{\mathbb{R}^{N}} \frac{\rho_{n}^{\infty}}{u_{n}}v_{n} \varphi_{n}\;dx,\nonumber
	\end{eqnarray}
	where $\rho_{n}^{\infty} \in \partial_{t} F_{\infty}(x,u_{n})$, that is,
	\begin{equation}\label{A27}
	\left<v_{n}^{+}-v_{n}^{-}, \varphi_{n}\right>- \int_{\mathbb{R}^{N}}V_{\infty}(x)v_{n}\varphi_{n}\;dx- \int_{\mathbb{R}^{N}} \frac{\rho_{n}^{\infty}}{u_{n}}v_{n} \varphi_{n}\;dx\rightarrow 0.
	\end{equation}
	Let
	\begin{equation*}
	\tilde{v}_{n}(x)=v_{n}(x+z_{n})\;\;\mbox{and}\;\; \tilde{u}_{n}(x)=u_{n}(x+z_{n}).
	\end{equation*}
	Knowing $||\tilde{v}_{n}||=||v_{n}||=1$, then,  going to a subsequence if necessary, there exists $\tilde{v} \in H^{1}(\mathbb{R}^{N})$ such that
	$$
	\tilde{v}_{n}\rightharpoonup \tilde{v}\;\;\mbox{in}\;\;H^{1}(\mathbb{R}^{N})
	$$
	and so,  $\tilde{v}_{n}\rightarrow \tilde{v}$ in $L_{loc}^{2}(\mathbb{R}^{N}).$ Since
	\begin{equation*}
	\int_{B_{r}(0)}|\tilde{v}|^{2}\;dx=\limsup_{n \to +\infty}\int_{B_{r}(0)}|\tilde{v}_{n}|^{2}\;dx=\limsup_{n \to +\infty}\int_{B_{r}(z_{n})}|v_{n}|^{2}\;dx\geq \eta,
	\end{equation*}
	it follows that $\tilde{v}\neq 0$.
	\begin{claim} \label{A46}
		For $\rho_{n}^{\infty} \in \partial_{t}F_{\infty}(x, u_{n})$,
		$$\lim_{n \to +\infty}\int_{\mathbb{R}^{N}} \frac{\rho_{n}^{\infty}}{u_{n}}v_{n}\varphi_{n}\;dx=0.$$
	\end{claim}
The claim follows from the limits below 
	\begin{equation*}
	\lim_{n \to +\infty}\int_{\mathbb{R}^{N}} \frac{|\underline{f}_{\infty}(x,u_{n})|}{|u_{n}|}|v_{n}\varphi_{n}|\;dx=0\;\;\mbox{and}\;\;\lim_{n \to +\infty}\int_{\mathbb{R}^{N}} \frac{|\overline{f}_{\infty}(x,u_{n})|}{|u_{n}|}|v_{n}\varphi_{n}|\;dx=0,
	\end{equation*}
	because
	\begin{equation*}
	|\rho_{n}^{\infty}| \leq |\overline{f}_{\infty}(x,u(x))|+|\underline{f}_{\infty}(x, u(x)) |,
	\end{equation*}
leads to
	\begin{equation*}
	\frac{|\rho_{n}^{\infty}|}{|u_{n}|} |v_{n}\varphi_{n}|\leq \frac{|\overline{f}_{\infty}(x,u(x))|}{|u_{n}|}|v_{n}\varphi_{n}|+\frac{|\underline{f}_{\infty}(x, u(x))|}{|u_{n}|} |v_{n}\varphi_{n}|.
	\end{equation*}
By $(H2)$, 
\begin{equation}\label{A57}
	\lim_{|t|\rightarrow +\infty} \frac{\underline{f}_{\infty}(x,t)}{t}=0\;\; \mbox{and}\;\; \lim_{|t|\rightarrow +\infty} \frac{\overline{f}_{\infty}(x,t)}{t}=0\;\;\mbox{uniformly in}\;\;\mathbb{R}^{N}.
\end{equation}
and by  $(H3)$,
	\begin{equation} \label{XX0}
	|\overline{f}_{\infty}(x,t)|\leq c_{0}|t|\;\;\mbox{and}\;\; |\underline{f}_{\infty}(x,t)|\leq c_{0}|t|, \;\forall t \in \mathbb{R} \quad \mbox{and} \quad \forall\; x \in \mathbb{R}^{N},
	\end{equation}
for some $c_{0}>0$. 

	Notice 
	\begin{equation*}
	\int_{\mathbb{R}^{N}} \frac{|\overline{f}_{\infty}(x,u(x))|}{|u_{n}(x)|}|v_{n}(x)\varphi_{n}(x)|dx= \int_{\mathbb{R}^{N}} \frac{|\overline{f}_{\infty}(x+z_{n}, \tilde{u}_{n}(x))|}{|\tilde{u}_{n}(x)|}|\tilde{v}_{n}(x)\varphi(x)|dx
	\end{equation*}	
	and
	\begin{equation*}
	\frac{|\overline{f}_{\infty}(x+z_{n}, \tilde{u}_{n}(x))|}{|\tilde{u}_{n}(x)|}|\tilde{v}_{n}(x)\varphi(x)|\leq c_{0}  |\tilde{v}_{n}(x)|\;|\varphi(x)|. 
	\end{equation*}
	Furthermore, fixing  $\omega_{n}(x)=|\tilde{v}_{n}(x)|\;|\varphi(x)|$ and $\omega(x)=|\tilde{v}(x)|\;|\varphi(x)|$, we have
	\begin{equation*}
	\sup_{n}||\omega_{n}||_{L^{1}(\mathbb{R}^{N})}<\infty,\;\int_{\mathbb{R}^{N}} \omega_{n}(x)dx\rightarrow \int_{\mathbb{R}^{N}} \omega(x) dx
	\end{equation*}
	and $\omega_{n}(x)\rightarrow \omega(x)\;\; \mbox{a.e in}\;\; \mathbb{R}^{N}$. So,
	\begin{equation*}
	\omega_{n}\rightarrow \omega\;\;\mbox{in}\;\;L^{1}(\mathbb{R}^{N}).
	\end{equation*}
	Going to a subsequence if necessary, there exists $h \in L^{1}(\mathbb{R}^{N})$ such that
	\begin{equation*}
	|\tilde{v}_{n}(x)|\;|\varphi(x)| \leq h(x)\;\;\mbox{a.e in }\;\;\mathbb{R}^{N},\;\forall\; n \in \mathbb{N}
	\end{equation*}
and 
	\begin{equation*}
	\left| \frac{\overline{f}_{\infty}(x+z_{n}, \tilde{u}_{n}(x))}{\tilde{u}_{n}(x)}\right||\tilde{v}_{n}(x)\varphi(x)|\leq c_{0} h(x),\;\;\mbox{a.e in}\;\;\mathbb{R}^{N},\;\forall\; n \in \mathbb{N}. 
	\end{equation*}
	Now, let us consider the sets
	$$A_{0}=\{x\in \mathbb{R}^{N}\;:\;\tilde{v}(x)=0\}\;\;\mbox{and}\;\;A=\{x\in \mathbb{R}^{N}\;:\;\tilde{v}(x)\neq 0\}.$$
	Thereby
	\begin{equation*}
	\left|\frac{\overline{f}_{\infty}(x+z_{n}, \tilde{u}_{n}(x))}{\tilde{u}_{n}(x)}\right|\;|\tilde{v}_{n}(x)\varphi(x)|dx\leq c_{0} |\tilde{v}_{n}(x)\varphi(x)|dx\rightarrow 0\;\;\mbox{a.e in}\;\;A_{0}.
	\end{equation*}
	By Lebesgue's dominated convergence theorem,
	\begin{equation*}
	\lim_{n \to +\infty}\int_{A_{0}} \left|\frac{\overline{f}_{\infty}(x+z_{n}, \tilde{u}_{n}(x))}{\tilde{u}_{n}(x)}\right|\;|\tilde{v}_{n}(x)\varphi(x)|dx=0.
	\end{equation*}
	Using the fact that $\tilde{u}_{n}(x)= \tilde{v}_{n}(x)||u_{n}||\rightarrow +\infty$ in $A$, by (\ref{A57}), 
	$$
	\lim_{n \to +\infty}\left[\left|\frac{\overline{f}_{\infty}(x+z_{n}, \tilde{u}_{n}(x))}{\tilde{u}_{n}(x)}\right|\;|\tilde{v}_{n}(x)\varphi(x)|\right]=0.\;\;\mbox{a.e in}\;\;A.
	$$
	Again, by Lebesgue's dominated convergence theorem, 
	\begin{equation*}
	\lim_{n \to +\infty}\int_{A} \left|\frac{\overline{f}_{\infty}(x+z_{n}, \tilde{u}_{n}(x))}{\tilde{u}_{n}(x)}\right|\;|\tilde{v}_{n}(x)\varphi(x)|dx=0.
	\end{equation*}
	Thus,
	\begin{equation*}
	\int_{\mathbb{R}^{N}}\left|\frac{\overline{f}_{\infty}(x+z_{n}, \tilde{u}_{n}(x))}{\tilde{u}_{n}(x)}\right|\;|\tilde{v}_{n}(x)\varphi(x)|dx=0.
	\end{equation*}
	Analogously
	\begin{equation*}
	\lim_{n \to +\infty}\int_{\mathbb{R}^{N}} \left|\frac{\underline{f}_{\infty}(x+z_{n}, \tilde{u}_{n}(x))}{\tilde{u}_{n}(x)}\right|\;|\tilde{v}_{n}(x)\varphi(x)|dx=0,
	\end{equation*}
	proving the claim
	
	By Claim \ref{A46},
	$$
	\lim_{n \to +\infty} \left(\int_{\mathbb{R}^{N}}V_{\infty}(x)v_{n}\varphi_{n}\;dx- \int_{\mathbb{R}^{N}} \frac{\rho_{n}^{\infty}}{u_{n}}v_{n} \varphi_{n}\;dx\right)=\int_{\mathbb{R}^{N}}V_{\infty}(x)\tilde{v}\varphi\;dx.
	$$
	Then, from (\ref{A27}), 
	\begin{equation*}
	\left<\tilde{v}^{+}-\tilde{v}^{-}, \varphi\right>- \int_{\mathbb{R}^{N}}V_{\infty}(x)\tilde{v}\varphi\;dx=0,\;\forall\;\varphi \in C_{0}^{\infty}(\mathbb{R}^{N}),
	\end{equation*}
	that is, $\tilde{v}\in H^{1}(\mathbb{R}^{N})\backslash\{0\}$ verifies
	$$-\Delta \tilde{v}+(V-V_{\infty})\tilde{v}=0,\;\;\mbox{in}\;\;\mathbb{R}^{N}.$$
	However, since  $(V-V_{\infty})$ is periodic, the spectrum of $-\Delta+V-V_{\infty}$ is absolutely continuous, then it has no eigenvalues [see \cite{Kuchment}, Theorem 4.59]. This shows that $(v_{n})$ cannot be non-vanishing.

	Suppose that $(v_{n})$ is vanishing. As in (\ref{A27}), 
	\begin{equation}\label{A28}
	||v_{n}^{+}||^{2}- \int_{\mathbb{R}^{N}} \frac{\rho_{n}}{u_{n}}v_{n} v^{+}_{n}\;dx=\left<v_{n}^{+}-v_{n}^{-}, v^{+}_{n}\right>- \int_{\mathbb{R}^{N}} \frac{\rho_{n}}{u_{n}}v_{n} v^{+}_{n}\;dx\rightarrow 0
	\end{equation}
	and
	\begin{equation}\label{A29}
	-||v_{n}^{-}||^{2}- \int_{\mathbb{R}^{N}} \frac{\rho_{n}}{u_{n}}v_{n} v^{-}_{n}\;dx=\left<v_{n}^{+}-v_{n}^{-}, v^{-}_{n}\right>- \int_{\mathbb{R}^{N}} \frac{\rho_{n}}{u_{n}}v_{n} v^{-}_{n}\;dx\rightarrow 0.
	\end{equation}
Since  $||v_{n}||=1$, (\ref{A28}) and (\ref{A29}) lead to 
	\begin{eqnarray}
	1-\int_{\mathbb{R}^{N}} \frac{\rho_{n}}{u_{n}}v_{n}( v^{+}_{n}-v^{-}_{n})\;dx\rightarrow 0,\nonumber
	\end{eqnarray}
	that is,
	\begin{equation} \label{ZZ1*}
	\lim_{n \to +\infty}\left(\int_{\mathbb{R}^{N}} \frac{\rho_{n}}{u_{n}}v_{n}( v^{+}_{n}-v^{-}_{n})\;dx\right)=1.
	\end{equation}
	By the definition of $\mu_{0}, \mu_{1}$ and $\mu_{-1}$, if $u \in H^{1}(\mathbb{R}^{N})$,
	\begin{equation*}\label{A51}
	||u^{+}||^{2}\geq \mu_{1} ||u^{+}||_{2}^{2},\;\forall\; u^{+} \in E^{+}\;\;\mbox{and}\;\;  ||u^{-}||^{2}\geq -\mu_{-1} ||u^{-}||_{2}^{2},\;\forall\; u^{-} \in E^{-}.
	\end{equation*}
	Hence,	
	\begin{equation}\label{A33}
	||u||^{2}\geq \mu_{0} ||u||_{2}^{2},\;\forall\; u \in H^{1}(\mathbb{R}^{N}).
	\end{equation}
	In what follows, let us consider the set
	$$
	\Omega_{n}=\left\{x\in\mathbb{R}^{N}\;:\; \frac{\rho_{n}(x)}{u_{n}(x)}\leq \mu_{0}-\delta \right\},
	$$
	where $\delta>0$ was given in $(H5)$. 
	
	By H\"older inequality, (\ref{A33}) and the orthogonality of $v_{n}^{+}$ and $v_{n}^{-}$ in $L^{2}(\mathbb{R}^{N})$, it follows that
	\begin{eqnarray}
	\int_{\Omega_{n}} \frac{\rho_{n}}{u_{n}}v_{n}(v_{n}^{+}-v_{n}^{-})\;dx& \leq & (\mu_{0}-\delta) \int_{\mathbb{R}^{N}} |v_{n}|\;|v_{n}^{+}-v_{n}^{-}|\;dx \nonumber \\
	&= &(\mu_{0}-\delta) ||v_{n}||_{2}\;||v_{n}^{+}-v_{n}^{-}||_{2} \nonumber \\
	&=&
	(\mu_{0}-\delta) ||v_{n}||_{2}^{2} \leq \frac{(\mu_{0}-\delta)}{\mu_{0}}<1. \nonumber
	\end{eqnarray}
This combined with (\ref{ZZ1*}) provides
	\begin{equation}\label{A31}
	\liminf_{n \to +\infty} \int_{\mathbb{R}^{N}\backslash \Omega_{n}} \frac{\rho_{n}}{u_{n}}v_{n}(v_{n}^{+}-v_{n}^{-})\;dx>0.
	\end{equation}
	\begin{claim}\label{A32}
		$$\lim_{n \to +\infty} |\mathbb{R}^{N}\backslash \Omega_{n}|=+\infty.$$
	\end{claim}
	Suppose that
	$$\limsup_{n \to +\infty} |\mathbb{R}^{N}\backslash \Omega_{n}|<\infty.$$
	Fixed $p\in (2,2^{*})$, the limit (\ref{A31}) combines with H\"older inequality to give  
	\begin{eqnarray}
	0& < &  \liminf_{n}\int_{\mathbb{R}^{N}\backslash \Omega_{n}} \frac{\rho_{n}}{u_{n}}v_{n}(v_{n}^{+}-v_{n}^{-})\;dx \nonumber \\
	&\leq & c_{0} \liminf_{n \to +\infty} \int_{\mathbb{R}^{N}\backslash \Omega_{n}} |v_{n}||v_{n}^{+}-v_{n}^{-}|\;dx \nonumber \\
	&\leq &  c_{0} \liminf_{n \to +\infty} \left[(|\mathbb{R}^{N}\backslash \Omega_{n}|)^{\frac{p-2}{p}}\left(\int_{\mathbb{R}^{N}\backslash \Omega_{n}}|v_{n}|^{\frac{p}{2}}|v_{n}^{+}-v_{n}^{-}|^{\frac{p}{2}}\;dx\right)^{\frac{2}{p}}\right] \nonumber \\
	&\leq &  c_{0} \liminf_{n \to +\infty} \left[(|\mathbb{R}^{N}\backslash \Omega_{n}|)^{\frac{p-2}{p}}\left(\int_{\mathbb{R}^{N}\backslash \Omega_{n}}|v_{n}|^{p}\;dx\right)^{\frac{1}{p}} \left(\int_{\mathbb{R}^{N}\backslash \Omega_{n}}|v_{n}^{+}-v_{n}^{-}|^{p}\;dx\right)^{\frac{1}{p}}\right]\nonumber \\
	&\leq &\tilde{c}_{0} \liminf_{n \to +\infty} \left[(|\mathbb{R}^{N}\backslash \Omega_{n}|)^{\frac{p-2}{p}}\left(\int_{\mathbb{R}^{N}\backslash \Omega_{n}}|v_{n}|^{p}\;dx\right)^{\frac{1}{p}} \right]. \nonumber
	\end{eqnarray}
	As we are supposing that $(v_{n})$ is vanishing, $v_{n}\rightarrow 0$ in $L^{s}(\mathbb{R}^{N})$ for $s \in(2,2^{*})$, and so,   
	$$
	0<\tilde{c}_{0} \liminf_{n \to +\infty} \left[(|\mathbb{R}^{N}\backslash \Omega_{n}|)^{\frac{p-2}{p}}\left(\int_{\mathbb{R}^{N}\backslash \Omega_{n}}|v_{n}|^{p}\;dx\right)^{\frac{2}{p}}\right]\rightarrow 0,
	$$
	which is absurd, and the Claim \ref{A32} is proved. Accordingly to $(H4)$, $(H5)$ and Claim \ref{A32}, 
	\begin{eqnarray}
	\int_{\mathbb{R}^{N}}\left(\frac{1}{2}\rho_{n}u_{n}-F(x,u_{n})\right)dx &\geq & \int_{\mathbb{R}^{N}\backslash \Omega_{n}}\left(\frac{1}{2}\rho_{n}u_{n}-F(x,u_{n})\right)dx \nonumber \\
	&\geq& \int_{\mathbb{R}^{N}\backslash \Omega_{n}} \delta\;dx\rightarrow +\infty, \nonumber
	\end{eqnarray}
	that is,
	$$\int_{\mathbb{R}^{N}}\left(\frac{1}{2}\rho_{n}u_{n}-F(x,u_{n})\right)dx\rightarrow +\infty.$$
	On the other hand, since $(u_{n})$ is a sequence $(C)_{c}$ and $\left<w_{n},u_{n}\right>\rightarrow 0$, we find 
	$$
	\int_{\mathbb{R}^{N}}\left(\frac{1}{2}\rho_{n}u_{n}-F(x,u_{n})\right)dx=I(u_{n})-\frac{1}{2}\left<w_{n},u_{n}\right>\rightarrow c,
	$$
which is a contradiction. This completes the proof. 
\end{proof}
Now, we are ready to conclude the proof of Theorem \ref{Teorema1}. \\

\noindent {\bf Proof of Theorem \ref{Teorema1}:} \,\, 

By Fatou's lemma the functional $I$ is $\tau$-upper semicontinuous, see \cite[Lemma 6.15]{MW}. Therefore, by Lemmas \ref{A25} and \ref{A22}, the functional $I$ satisfies the hypotheses of Theorem \ref{A85}, and so, by Lemma \ref{A26} there exists a bounded  $(C)_{c}$ sequence for the functional $I$, denoted by $(u_{n})\subset H^{1}(\mathbb{R}^{N})$,  i.e, 
\begin{equation*}
I(u_{n})\rightarrow c,\;\;\; (1+||u_{n}||)\lambda_{I}(u_{n})\rightarrow 0\;\;\mbox{and}\;\; ||u_{n}||\leq K,\;\forall\;n \in \mathbb{N},
\end{equation*}
for some $K>0$. By \cite [Lemma 1.21]{MW} and (\ref{P1}), there exists $\delta_1>0$ such that
$$\liminf_{n \to +\infty} \sup_{y \in \mathbb{R}^{N}} \int_{B(y,1)} |u_{n}|^{2}dx\geq \delta_1.$$
In addition, there is  $(z_{n}) \subset \mathbb{Z}^{N}$ such that
\begin{equation*}
\int_{B(z_{n},1+\sqrt{N})} |u_{n}|^{2} dx \geq \frac{\delta_1}{4}, \;n \geq n_{0}.
\end{equation*}
Setting $v_{n}(x)=u_{n}(x+z_{n})$, we compute 
\begin{equation}\label{N1}
\int_{B(0,1+\sqrt{N})} |v_{n}(x)|^{2}dx= \int_{B(z_{n},1+\sqrt{N})} |u_{n}(x)|^{2} dx\geq \frac{\delta_1}{4},\;n \geq n_{0}.
\end{equation}
Moreover, a simple computation also shows  $(v_{n}) \subset H^{1}(\mathbb{R}^{N})$ is a $(PS)_{c}$ sequence for $I$ (see \cite{AG} for details) and $||v_{n}||=||u_{n}||$. Since $(u_{n})$ is bounded, going to a subsequence if necessary, $v_{n} \rightharpoonup v$ in $H^{1}(\mathbb{R}^{N})$ and by (\ref{N1}) $v \neq 0$. Now, by using the Proposition \ref{7}, we can argue as in \cite{AG} to conclude that
$$
-\Delta v(x) + V(x) v(x) \in \partial_u F(x,u)\;\;\mbox{a.e in}\;\; \mathbb{R}^{N},
$$
showing the desired result. 

\section{The non periodic case}

In this section, we will prove the Theorem \ref{Teorema2}. 
Similarly to Section 3, the $(H2)$ and $(H3)$ ensure that the energy functional associated with problem $({P})$ defined by  
$$
I(u)= \frac{1}{2}\int_{\mathbb{R}^{N}}(|\nabla u|^{2}+V(x)u^{2})dx-\int_{\mathbb{R}^{N}}F(x,u)dx, \;u \in H^{1}(\mathbb{R}^{N}),
$$
is well defined. Furthermore, since the lemmas showed in Section 5 do not depend on the periodicity of function $f$, but only of its growth, all of them are also true in this section,  and we have the same link geometry. Consequently, there is a bounded sequence $(u_{n})\subset H^{1}(\mathbb {R}^{N})$ such that
\begin{equation*}
(1+||u_{n}||)\lambda_{I}(u_{n})\rightarrow 0\;\;\mbox{and}\;\;I(u_{n})\rightarrow c,
\end{equation*} 
where
$$
c=\inf_{\gamma \in \Gamma} \sup_{u \in \mathcal{M}}I(\gamma(1,u)).
$$

Next, we are going to recall some facts involving the periodic problem 
$$
\left\{\begin{aligned}
-\Delta u + V(x)u =& h(x,u),\;\;\mbox{ in}\;\;\mathbb{R}^{N},\;N\geq 3 \\
u \in H^{1}(\mathbb{R}^{N}),
\end{aligned}
\right. \leqno{(A)}
$$
where $ h:\mathbb{R}^{N}\times\mathbb{R}\rightarrow \mathbb{R} $ is the continuous function that satisfies the assumption $(H6)$. 

First of all, we would like point out that the energy functional associated with problem $(A)$, given by
$$
{J}(u)= \frac{1}{2}\int_{\mathbb{R}^{N}}(|\nabla u|^{2}+V(x)u^{2})dx-\int_{\mathbb{R}^{N}}H(x,u)dx, \;u \in H^{1}(\mathbb{R}^{N}),
$$
is well defined and ${J} \in C^{1}(H^{1}(\mathbb{R}^N),\mathbb{R})$, where $H(x,t)= \int_{0}^{t}h(x,s)ds$.

It is proved in \cite{G.B Li} that the functional $J$ has a nontrivial critical $ u_{0}\in H^{1}(\mathbb{R}^{N})\backslash\{0\}$ that satisfies: 
\begin{equation*}
\tau^{2} ||u_{0}^{+}||^{2}-||y||^{2}- \int_{\mathbb{R}^{N}}V_{\infty}(x)(\tau u_{0}^{+}+y)^{2}dx<0, \quad \forall \tau\in \mathbb{R} \quad \mbox{and} \quad y \in X^{-},
\end{equation*}
\begin{equation} \label{T21}
{J}(u_{0})= \min\{{J}(u)\;:\;u \neq 0\;\;\mbox{and}\;\; {J}'(u)=0\},
\end{equation}
and
$$
\sup_{u \in \mathcal{M}}{J}(u)={J}(u_{0})>0.
$$

In particular, if we choose $ z_{0}=u_{0}^{+}$ in the definition of $\mathcal{M}$, see Section 4, the Lemma \ref{A22} still holds with $z_0=u_0^+$. 

\vspace{0.5 cm}

Using the above information, we are ready to conclude the proof of Theorem \ref{Teorema2}.\\
 
\noindent {\bf Proof of Theorem \ref{Teorema2}:} \\

From above commentaries, there is a bounded sequence $(u_{n})\subset H^{1}(\mathbb{R}^{N})$ satisfying
\begin{equation*}
(1+||u_{n}||)\lambda_{J}(u_{n})\rightarrow 0\;\;\mbox{and}\;\;J(u_{n})\rightarrow c
\end{equation*} 
where
$$
c=\inf_{\gamma \in \Gamma} \sup_{u \in \mathcal{M}}J(\gamma(1,u)).
$$

Therefore, there are $w_{n} \in \partial I(u_{n})$ with $\lambda_{I}(u_{n})=||w_{n}||_{*}$ and $\rho_{n} \in \partial \Psi(u_{n})$ such that
\begin{equation*}
	\left<w_{n}, \varphi \right>= \left<Q'(u_{n}), \varphi \right>- \left<\rho_{n}, \varphi \right>,\;\forall\;\varphi \in  X,\;\forall\; n \in \mathbb{N}.
\end{equation*}
 
Since $H^{1}(\mathbb{R}^{N})$ is reflexive, going to a subsequence if necessary, there is $u \in H^{1}(\mathbb{R}^{N})$ such that
\begin{equation*}
u_{n}\rightharpoonup u\;\;\mbox{in}\;\;H^{1}(\mathbb{R}^{N}).
\end{equation*}
If $u\neq 0$, then the Theorem \ref{Teorema2} is proved. If $u=0$, we have
\begin{equation*}
u_{n}\rightharpoonup 0\;\;\mbox{in}\;\;H^{1}(\mathbb{R}^{N}).
\end{equation*}
From $(H7)$, $G(x,t)\geq H(x,t)$ for all $t \in \mathbb{R}$, and so, 
\begin{equation*}
0<c:=\inf_{\gamma \in \Gamma} \sup_{u \in \mathcal{M}}I(\gamma(1,u))\leq \sup_{u \in \mathcal{M}}I(u) \leq \sup_{u \in \mathcal{M}}{J}(u)={J}(u_{0}) 
\end{equation*}
that is,
\begin{equation}\label{O25}
c \leq {J}(u_{0}).
\end{equation}
Next, we are going to prove that $c=J(u_0)$. 
\begin{claim}\label{A67}
$$
\rho_{n}-\tilde{\Psi}'(u_{n})\rightarrow 0 \;\;\mbox{and}\;\;\ 
	\Psi(u_{n})-\tilde{\Psi}(u_{n})\rightarrow 0,
$$ 
where
\begin{equation*}
\Psi(u)=\int_{\mathbb{R}^{N}} F(x,u)dx\;\;\mbox{and}\;\;\tilde{\Psi}(u)=\int_{\mathbb{R}^{N}} H(x,u)dx.
\end{equation*}
\end{claim}
Let $\varphi \in H^{1}(\mathbb{R}^{N})$ with $\|\varphi\|\leq 1$. By $(H7)$,  H\"older inequality and using the inequality $||u||^{2}\geq \mu_{0} ||u||_{2}^{2} \;\forall\; u \in H^{1}(\mathbb{R}^{N})$, we obtain
\begin{eqnarray}
\left|\left<\rho_{n}-\tilde{\Psi}'(u_{n}), \varphi \right>\right|&=&\left|\int_{\mathbb{R}^{N}} [\rho_{n}-h(x,u_{n})]\varphi dx\right| \nonumber \\
&\leq& \int_{\mathbb{R}^{N}} a(x)|u_{n}|\;|\varphi| dx \nonumber \\
&\leq & \left(\int_{\mathbb{R}^{N}}|a(x)|^{2}|u_{n}|^{2}dx\right)^{\frac{1}{2}} ||\varphi||_{2}\nonumber \\
&\leq &\left(\frac{1}{\mu_{0}}\right)^{\frac{1}{2}}\left(\int_{\mathbb{R}^{N}}|a(x)|^{2}|u_{n}|^{2}dx\right)^{\frac{1}{2}}  ||\varphi||.\nonumber 
\end{eqnarray}
Using the fact that $a(x)\rightarrow 0$ whenever $|x|\rightarrow +\infty$, given $\varepsilon>0$ there is  $R_{0}>0$ such that $|a(x)|\leq \varepsilon$ for $|x|> R_{0}$. By compact embedding $H^{1}(B_{R_{0}})\hookrightarrow L^{2}(B_{R_{0}})$ we have $u_{n}\rightarrow 0$ in $ L^{2}(B_{R_{0}})$, thereby there is $n_{0}\in \mathbb{N}$ such that $||u_{n}||_{L^{2}(B_{R_{0}})}\leq \varepsilon,\;\forall\;n \geq n_{0}$, and so,
\begin{eqnarray}
\int_{\mathbb{R}^{N}}|a(x)|^{2}|u_{n}|^{2}dx&=& \int_{B_{R_{0}}}|a(x)|^{2}|u_{n}|^{2}dx+\int_{B^{c}_{R_{0}}}|a(x)|^{2}|u_{n}|^{2} dx\nonumber \\
&\leq & ||a||_{\infty}^{2}\int_{B_{R_{0}}}|u_{n}|^{2}dx+\varepsilon^{2}\int_{B^{c}_{R_{0}}}|u_{n}|^{2}dx \nonumber \\
&\leq &\varepsilon^{2}(||a||_{\infty}^{2}+K^{2}), \nonumber
\end{eqnarray}
where $||u_{n}||\leq K$ for all $n \in \mathbb{N}.$ As $\epsilon$ is arbitrary,  
$$
\rho_{n}-\tilde{\Psi}'(u_{n})\rightarrow 0 \quad \mbox{in} \quad (H^{1}(\mathbb{R}^N))^*.
$$
A similar argument guarantees that  
$$
\Psi(u_{n})-\tilde{\Psi}(u_{n})\rightarrow 0 \quad \mbox{in} \quad \mathbb{R}.
$$
Then, by Claim \ref{A67},
\begin{equation*}
J'(u_{n}) \rightarrow 0\;\;\mbox{and}\;\;{J}(u_{n}) \rightarrow c, \;\;\mbox{as}\;\;n\rightarrow +\infty.
\end{equation*}
As $(u_{n})$ is bounded, it follows that $(u_{n})$ is a  $(C)_{c}$-sequence for the functional ${J}$, and so, arguing as in the proof of Theorem \ref{Teorema1},  there is a nontrivial critical point $u_1\neq 0$ of ${J}$, with ${J}(u_1)\leq c$. On the other hand, by (\ref{T21}), we must have ${J}(u_{0})\leq {J}(u_1)$, from where it follows that 
\begin{equation*}
c \leq \sup_{u \in \mathcal{M}}I(u) \leq {J}(u_{0})\leq {J}(u_1) \leq c,
\end{equation*}
that is, 
\begin{equation*}
\sup_{u \in \mathcal{M}}I(u)=c.
\end{equation*}
Now, as Corollary \ref{N4} still holds when $f$ is non periodic, we can conclude that there is $v\in H^{1}(\mathbb{R}^{N})$ such that $0 \in \partial I(v)$ and $I(v)=c>0$. This finishes the proof of Theorem \ref{Teorema2}

\end{document}